\documentclass[oneside,english]{amsart}
\usepackage[T1]{fontenc}
\usepackage[latin9]{inputenc}
\usepackage{amstext}
\usepackage{amsthm}
\usepackage{amssymb}

\makeatletter
\numberwithin{equation}{section}
\numberwithin{figure}{section}
\theoremstyle{plain}
\newtheorem{thm}{\protect\theoremname}[section]
\theoremstyle{definition}
\newtheorem{defn}[thm]{\protect\definitionname}
\theoremstyle{plain}
\newtheorem{lem}[thm]{\protect\lemmaname}
\theoremstyle{plain}
\newtheorem{cor}[thm]{\protect\corollaryname}
\theoremstyle{plain}
\newtheorem{prop}[thm]{\protect\propositionname}
\theoremstyle{remark}
\newtheorem{rem}[thm]{\protect\remarkname}

\usepackage{alex}

\makeatother

\usepackage{babel}
\providecommand{\corollaryname}{Corollary}
\providecommand{\definitionname}{Definition}
\providecommand{\lemmaname}{Lemma}
\providecommand{\propositionname}{Proposition}
\providecommand{\remarkname}{Remark}
\providecommand{\theoremname}{Theorem}

\begin{document}
\title{Representations of $\SL_{n}$ over finite local rings of length two}
\author{Alexander Stasinski}
\address{Alexander Stasinski, Department of Mathematical Sciences, Durham University,
Durham, DH1 3LE, UK}
\email{alexander.stasinski@durham.ac.uk}
\begin{abstract}
Let $\F_{q}$ be a finite field of characteristic $p$ and let $W_{2}(\F_{q})$
be the ring of Witt vectors of length two over $\F_{q}$. We prove
that for any integer $n$ such that $p$ divides $n$, the groups
$\SL_{n}(\F_{q}[t]/t^{2})$ and $\SL_{n}(W_{2}(\F_{q}))$ have the
same number of irreducible representations of dimension $d$, for
each $d$.
\end{abstract}

\maketitle

\section{Introduction\label{sec:Introduction}}

Let $\cO$ be a complete discrete valuation ring with maximal ideal
$\mfp$ and residue field $\F_{q}$ with $q$ elements and characteristic
$p$. For an integer $r\geq1$, we write $\cO_{r}=\cO/\mfp^{r}$.
It is known that $\cO_{2}$, and in fact any finite local ring of
length two with residue field $\F_{q}$, is isomorphic to either the
ring $W_{2}(\F_{q})$ of Witt vectors of length two or $\F_{q}[t]/t^{2}$
(see \cite[Lemma~2.1]{Stasinski-Vera-Gajardo}). For a finite group
$G$ and an integer $d\geq1$, let $\Irr_{d}(G)$ denote the set of
isomorphism classes of irreducible complex representations of $G$
of dimension~$d$.

P.~Singla \cite{Pooja-classicalgrps} has proved that when $p$ does
not divide $n$, we have

\[
\#\Irr_{d}(\SL_{n}(\F_{q}[t]/t^{2}))=\#\Irr_{d}(\SL_{n}(W_{2}(\F_{q}))),
\]
for all $d\geq1$. In \cite{Stasinski-Vera-Gajardo} a new proof of
this was given, as well as a generalisation when $\SL_{n}$ is replaced
by any reductive group scheme $\G$ over $\Z$ (with connected fibres)
such that $p$ is very good for $\G\times_{\Z}\F_{q}$.

The case $\G=\SL_{n}$ with $p\mid n$ was also studied in \cite{Pooja-classicalgrps}
but the argument there remains incomplete (see \cite[Section~5]{Stasinski-Vera-Gajardo}).
In the present paper, we complete the argument and prove that for
all $n$ such that $p\mid n$, we have
\[
\#\Irr_{d}(\SL_{n}(\F_{q}[t]/t^{2}))=\#\Irr_{d}(\SL_{n}(W_{2}(\F_{q}))),
\]
for all $d\geq1$. This is Theorem~\ref{thm:Main}, whose proof is
finished in Section~\ref{sec:The-main-result}.

The main new ingredient in the proof is the following. Let $s:\M_{n}(\F_{q})\rightarrow\M_{n}(\cO_{2})$
be the function induced by the multiplicative section $s:\F_{q}^{\times}\rightarrow\cO_{2}^{\times}$.
We show (see Theorem~\ref{thm:Thm1}) that for any $x\in\M_{n}(\F_{q})$
in \emph{Weyr normal form} the reduction mod $\mfp$ map on centralisers
\[
C_{\SL_{n}(\cO_{2})}(s(x))\longrightarrow C_{\SL_{n}(\F_{q})}(x)
\]
is surjective. The Weyr normal form of a matrix is a kind of dual
of the more common Jordan normal form, but has the advantage that
centralisers are block upper triangular (see Section~\ref{sec:Surjectivity-of-the}).
By contrast, we do not know how to prove a statement like this when
`Weyr' is replaced by `Jordan'.

Another important ingredient of the proof of Theorem~\ref{thm:Main}
is Lemma~\ref{lem:V-complement} which gives the precise structure
of the quotient $C_{\SL_{n}(\F_{q})}(x+Z)/C_{\SL_{n}(\F_{q})}(x)$,
where $Z$ is the centre of $\M_{n}(\F_{q})$ and $\SL_{n}(\F_{q})$
acts by conjugation on $\M_{n}(\F_{q})/Z$. It turns out that this
quotient is cyclic, and is generated by a coset $vC_{\SL_{n}(\F_{q})}(x)$,
where $v$ is a permutation matrix.

Of course, Theorem~\ref{thm:Main} is trivially true whenever $\SL_{n}(\F_{q}[t]/t^{2})\cong\SL_{n}(W_{2}(\F_{q}))$.
However, this is almost never the case, as shown by Sah \cite{Sah-I},
namely, when $p=q$, the groups are isomorphic only for $(n,p)\in\{(2,3),(3,2)\}$.
In Section~\ref{sec:The-isomorphism-problem} we give a new proof,
following a suggestion of Y.~de~Cornulier, that the groups $\SL_{n}(\F_{q}[t]/t^{2})$
and $\SL_{n}(W_{2}(\F_{q}))$ are not isomorphic when $p\geq5$.

\section{\label{sec:Notational-preliminaries}Notational preliminaries}

Define the groups
\begin{alignat*}{2}
G_{2} & =\GL_{n}(\cO_{2}),\qquad & G & =\GL_{n}(\F_{q}),\\
S_{2} & =\SL_{n}(\cO_{2}),\qquad & S & =\SL_{n}(\F_{q}).
\end{alignat*}
The reduction map $\cO_{2}\rightarrow\F_{q}$ induces surjective homomorphisms
$\rho:G_{2}\rightarrow G$ and $\rho:S_{2}\rightarrow S$ (the surjectivity
follows either from smoothness of the group schemes $\GL_{n}$ and
$\SL_{n}$ or by noting that a set of generators of $G$ or $S$ can
be lifted to $G_{2}$ or $S_{2}$, respectively). We let $G^{1}$
and $S^{1}$ denote the kernels of the homomorphisms $\rho$, respectively.
Similarly, for any matrix $x\in\M_{n}(\cO_{2})$ we will denote its
image in $\M_{n}(\F_{q})$ by $\rho(x)$.

Let $\M_{n}(\F_{q})=\Lie(\GL_{n})(\F_{q})$ be the ring of $n\times n$
matrices over $\F_{q}$ and let $\M_{n}^{0}(\F_{q})=\Lie(\SL_{n})(\F_{q})=\{x\in\M_{n}(\F_{q})\mid\Tr(x)=0\}$.
Choosing a prime element $\varpi\in\cO_{2}$, the map $1+\varpi x\mapsto\rho(x)$
induces isomorphisms $G^{1}\cong\M_{n}(\F_{q})$ and $S^{1}\cong\M_{n}^{0}(\F_{q})$.
Consider the $G$-equivariant map
\begin{align*}
\M_{n}(\F_{q}) & \longrightarrow\M_{n}^{0}(\F_{q})^{*}:=\Hom_{\F_{q}}(\M_{n}^{0}(\F_{q}),\F_{q}),\\
x\longmapsto f_{x} & \text{ where }f_{x}(y)=\Tr(xy),\quad\text{for }y\in\M_{n}^{0}(\F_{q})
\end{align*}
It is easy to see that the kernel of this map is the subalgebra $Z$
of scalar matrices. Since $\dim\M_{n}(\F_{q})=n^{2}$ and $\dim\M_{n}^{0}(\F_{q})^{*}=\dim\M_{n}^{0}(\F_{q})=n^{2}-1$,
this map is surjective, so we have a $G$-equivariant isomorphism
\[
\M_{n}(\F_{q})/Z\longiso\M_{n}^{0}(\F_{q})^{*}.
\]
Using this isomorphism, we will identify elements in $\M_{n}^{0}(\F_{q})^{*}$
with elements in $\M_{n}(\F_{q})/Z$.

For $x\in\M_{n}(\F_{q})$, let $\beta\in\M_{n}^{0}(\F_{q})^{*}$ be
the element corresponding to $x+Z$. As in \cite[Section~4.1]{Stasinski-Vera-Gajardo},
we then have a degree one character $\psi_{x+Z}=\psi_{\beta}\in\Irr(S^{1})$.
More precisely, for $\hat{y}\in\M_{n}(\cO_{2})$ with $\tr(\hat{y})=0$
and such that $\rho(\hat{y})=y$, we have $1+\varpi\hat{y}\in S^{1}$
and
\[
\psi_{x+Z}(1+\varpi\hat{y})=\psi(f_{x}(y))=\psi(\Tr(xy)).
\]
Similarly, we have an extension $\psi_{x}\in\Irr(G^{1})$ of $\psi_{x+Z}$,
given by the same formula but for $\hat{y}\in\M_{n}(\cO_{2})$. The
map $x+Z\mapsto\psi_{x+Z}$ induces a $G$-equivariant isomorphism
of abelian groups
\[
\M_{n}(\F_{q})/Z\longiso\Irr(S^{1}).
\]
Note that since $S_{2}$ is normal in $G_{2}$, the action of $S_{2}$
on $\Irr(S^{1})$ (which factors through $S$) extends to an action
of $G_{2}$ (which factors through $G$).

\section{\label{sec:Surjectivity-of-the}Surjectivity of the map on centralisers}

Let $n\geq2$ be a fixed integer. Let $s:\M_{n}(\F_{q})\rightarrow\M_{n}(\cO_{2})$
be the function induced by the multiplicative section $s:\F_{q}^{\times}\rightarrow\cO_{2}$. 

In this section, we prove one of the key results of the present paper,
namely:
\begin{thm}
\label{thm:Thm1}Assume that $x\in\M_{n}(\F_{q})$ is in Weyr normal
form. Then the map
\[
\rho:C_{S_{2}}(s(x))\longrightarrow C_{S}(x)
\]
 is surjective.
\end{thm}

The Weyr normal form of a matrix is a lesser known dual form of the
Jordan normal form. We will briefly introduce the Weyr normal form
following \cite{Weyr-form-book}, where much further information about
it can be found. While the Jordan normal form is a direct sum of Jordan
blocks (called basic Jordan matrices in \cite{Weyr-form-book}), possibly
with the same eigenvalues, the Weyr normal form is a direct sum of
``basic Weyr matrices'' with \emph{distinct eigenvalues}. The following
definition is a paraphrasing of \cite[Definition~2.1.1]{Weyr-form-book}.
\begin{defn}
Let $A$ be a commutative ring with identity. A matrix $W\in\M_{n}(A)$
is called a \emph{basic Weyr matrix} (with eigenvalue $\lambda\in A$)
if it is of the following form: there is a partition $n=n_{1}+\dots+n_{r}$
with $n_{1}\geq\dots\geq n_{r}\geq1$ such that, when $W$ is viewed
as an $r\times r$ blocked matrix $(W_{ij})$, where the $(i,j)$
block $W_{ij}$ is an $n_{i}\times n_{j}$ matrix, the following three
conditions hold:
\begin{enumerate}
\item $W_{ii}=\lambda I_{n_{i}}$, for each $i=1,\dots,r$ (where $I_{n_{i}}$
is the identity matrix of size $n_{i}$);
\item the superdiagonal blocks $W_{i,i+1}$ are full column rank $n_{i}\times n_{i+1}$
matrices in reduced row-echelon form (i.e., an identity matrix followed
by zero rows) for $i=1,\dots,r-1$;
\item all other blocks of $W$ are zero (i.e., $W_{ij}=0$ when $j\notin\{i,i+1\}$).
\end{enumerate}
\end{defn}

An example of a nilpotent basic Weyr matrix, corresponding to the
partition $7=3+2+2$, is
\[
\left[\begin{array}{ccc|cc|cc}
0 & 0 & 0 & 1 & 0\\
0 & 0 & 0 & 0 & 1\\
0 & 0 & 0 & 0 & 0\\
\hline  &  &  & 0 & 0 & 1 & 0\\
 &  &  & 0 & 0 & 0 & 1\\
\hline  &  &  &  &  & 0 & 0\\
 &  &  &  &  & 0 & 0
\end{array}\right],
\]
where omitted entries are zero. The Jordan normal form of this matrix
is given by Jordan blocks corresponding to the dual partition $7=3+3+1$.
This duality between the Jordan and Weyr forms is a general phenomenon
(see \cite{Weyr-form-book}).

A matrix in $\M_{n}(A)$, $A$ a commutative ring with identity, is
said to be a \emph{Weyr matrix} or is in \emph{Weyr (normal) form}
if it is a direct sum of basic Weyr matrices with pairwise distinct
eigenvalues (see \cite[Definition~2.1.5]{Weyr-form-book}). Note that
a Weyr matrix cannot have more than one basic Weyr block for a given
eigenvalue. This is in contrast to the Jordan normal form, where multiple
basic Jordan blocks with the same eigenvalue can appear. In particular,
the direct sum of two nilpotent basic Weyr matrices is not a matrix
in Weyr normal form.

Just as for the Jordan normal form, the Weyr form of a matrix $x\in\M_{n}(F)$
over a field $F$ such that the eigenvalues of $x$ lie in $F$, exists
and is unique (up to the order of the blocks. This implies that there
exists a finite field extension $L/F$ containing the eigenvalues
of $x$ and a $g\in\GL_{n}(L)$ such that $gxg^{-1}\in\M_{n}(L)$
is in Weyr normal form.

In contrast to the Jordan normal form, the Weyr form has the advantage
that the centraliser is block upper triangular. It is this fact about
centralisers of matrices in Weyr form which allows us to prove Theorem~\ref{thm:Thm1}.
It is possible that Theorem~\ref{thm:Thm1} holds also when ``Weyr''
is replaced by ``Jordan'', but we do not know how to prove this.

\subsection{Centralisers}

If $A$ is a ring, $n=n_{1}+\dots+n_{r}$ and $x_{i}\in\M_{n_{i}}(A)$,
we will write $x_{1}\oplus\dots\oplus x_{r}$ for the block diagonal
matrix $\diag(x_{1},\dots,x_{r})\in\M_{n}(A)$.

The following result, well-known in the case of fields, allows us
to reduce to the case of centralisers of matrices with a single eigenvalue.
The result holds, with essentially the same proof, over any principal
ideal local ring, but for notational simplicity we only state it over
$\cO_{2}$ (since we don't loose any generality, it then also holds
over $\tilde{\cO}_{2}$, defined with respect to any finite extension
$k/\F_{q}$).
\begin{lem}
\label{lem:C-decomp-in-blocks}Let $x=x_{1}\oplus\dots\oplus x_{r}\in\M_{n}(\cO_{2})$,
for $x_{i}\in\M_{n_{i}}(\cO_{2})$, such that all the eigenvalues
of $\rho(x)$ lie in the residue field $\F_{q}$, each $\rho(x_{i})$
has a single eigenvalue $\lambda_{i}$ in $\F_{q}$ and such that
$\lambda_{i}\neq\lambda_{j}$ when $i\neq j$. Then, as $\cO_{2}$-algebras,
\[
C_{\M_{n}(\cO_{2})}(x)\cong C_{\M_{n_{1}}(\cO_{2})}(x_{1})\times\dots\times C_{\M_{n_{r}}(\cO_{2})}(x_{r}),
\]
where the isomorphism, from right to left, is given by $(B_{1},\dots,B_{r})\mapsto B_{1}\oplus\dots\oplus B_{r}$.
\end{lem}

\begin{proof}
This is very similar to the well-known case for matrices over a field,
treated in \cite[Proposition~3.1.1]{Weyr-form-book}. The only difference
is that given a relation 
\[
x_{i}y_{ij}=y_{ij}x_{j},
\]
where $y_{ij}\in\M_{n_{i}\times n_{j}}(\cO_{2})$ is a block matrix
with the same block structure as $x$, we need to reduce the relation
mod $(\varpi)$ to obtain
\[
\rho(x_{i})\rho(y_{ij})=\rho(y_{ij})\rho(x_{j}).
\]
Then, as in the proof of \cite[Proposition~3.1.1]{Weyr-form-book},
Sylvester's theorem implies that $\rho(y_{ij})=0$, so $y_{ij}=\varpi y_{ij}'$,
for some $y_{ij}'\in\M_{n_{i}\times n_{j}}(\cO_{2})$. Thus $\varpi x_{i}y_{ij}'=\varpi y_{ij}'x_{j}$,
so $\rho(x_{i})\rho(y_{ij}')=\rho(y_{ij}')\rho(x_{j})$, and by Sylvester's
theorem, $y_{ij}'=\varpi y_{ij}''$, for some $y_{ij}''$; hence $y_{ij}=0$.
\end{proof}
For any ring $A$, and integers $1\leq i,j\leq n$ we let $E_{ij}(A)$
be the elementary subalgebra of $\M_{n}(A)$ consisting of matrices
whose $(i,j)$-entry is an arbitrary element of $A$ and has all other
entries zero.

We will often write a partition of $n$ as $(d_{1}^{e_{1}},\dots,d_{m}^{e_{m}})$,
$d_{i},e_{i}\in\N$, which means that $n=e_{1}d_{1}+\dots+e_{m}d_{m}$
and $d_{1}>d_{2}>\dots>d_{m}$. It is obvious that a basic Weyr matrix
has the same centraliser as the corresponding nilpotent matrix obtained
by replacing the diagonal entries by zeros. We thus focus on nilpotent
Weyr matrices. The explicit structure of the centraliser of a nilpotent
Weyr matrix over a field is given in \cite[Proposition~2.3.3]{Weyr-form-book},
and since the proof goes through over any ring, the following result
is an immediate consequence:
\begin{lem}
\label{lem:Weyr-centraliser}Let $A$ be a commutative ring with identity
and $x\in\M_{n}(A)$ a basic Weyr matrix, corresponding to a partition
$n=n_{1}+\dots+n_{r}$. There exists a partition $(d_{1},\dots,d_{m})$
of $n$, uniquely determined by $n_{1},\dots,n_{r}$, such that, as
$A$-algebras,
\[
C_{\M_{n}(A)}(x)\cong\Big(\prod_{\ell=1}^{m}\M_{d_{\ell}}(A)\Big)\oplus N(A),
\]
where the first summand is embedded as a block-diagonal subalgebra
and $N(A)$ is the direct sum of certain subalgebras $E_{ij}(A)$
such that whenever $n_{1}+\dots+n_{\ell}<i\leq n_{1}+\dots+n_{\ell+1}$,
for some $1\leq\ell\leq r-2$, we have $j>n_{1}+\dots+n_{\ell+1}$.

Moreover, the $(i,j)$ such that $E_{ij}(A)$ is a non-zero summand
of $N(A)$ are completely determined by the partition $n_{1},\dots,n_{m}$
(and hence independent of $A$).
\end{lem}

The conditions on $i$ and $j$ in the above lemma just say that $N(A)$
consists of $E_{ij}(A)$ having non-zero entries only above the block-diagonal.
Note that not all such subalgebras $E_{ij}(A)$ necessarily occur
in $N(A)$.

\subsection{\label{subsec:The-intersection-with}The intersection with $\SL_{n}$}

If $x\in\M_{n}(\F_{q})$ is a Weyr matrix with a single eigenvalue,
then so is $s(x)\in\M_{n}(\cO_{2})$. Since Theorem~\ref{thm:Thm1}
is concerned with $C_{\SL_{n}(\F_{q})}(x)$ and $C_{\SL_{n}(\cO_{2})}(s(x))$,
we will, in view of Lemma~\ref{lem:Weyr-centraliser}, consider the
product of the determinants of the block-diagonal subalgebra.

Let $F$ denote the field of fractions of $\cO$. For any partition
$\lambda=(d_{1}^{e_{1}},\dots,d_{m}^{e_{m}})$, $d_{i},e_{i}\in\N$
of $n$, let $X_{\lambda}$ be the affine group scheme over $\cO$
defined by
\[
X_{\lambda}(A)=\{((x_{ij}^{(1)}),\dots,(x_{ij}^{(m)}))\in\prod_{\ell=1}^{m}\M_{d_{\ell}}(A)\mid\prod_{\ell=1}^{m}\det(x_{ij}^{(\ell)})^{e_{\ell}}=1\},
\]
for any $\cO$-algebra $A$. Here for each $\ell$, we have $(x_{ij}^{(1)})\in\M_{d_{\ell}}(A)$
with $1\leq i,j,\leq d_{\ell}$. Note that $X_{\lambda}$ is not always
smooth over $\cO$: Take, for instance, $\cO$ with residue field
$\F_{2}$, $m=1$ and $d_{1}=1$, $e_{1}=2$; then $X\times_{\cO}\F_{2}=\Spec\F_{2}[x]/(x-1)^{2}$,
which is not reduced.
\begin{lem}
\label{lem:smooth-X_lambda}Let $\lambda=(d_{1}^{e_{1}},\dots,d_{m}^{e_{m}})$
and assume that $p\nmid e_{\ell}$ for some $\ell\in\{1,\dots,m\}$.
Then $X_{\lambda}$ is smooth over $\cO$.
\end{lem}

\begin{proof}
Since $X_{\lambda}$ is a hypersurface, the fibres $X_{\lambda}\times F$
and $X_{\lambda}\times\F_{q}$ have the same dimension, so $X_{\lambda}$
is flat over $\cO$ (this can also be seen by noting that since the
defining ideal is generated by $\prod_{\ell=1}^{m}\det(x_{ij}^{(\ell)})^{e_{\ell}}-1\not\in\mfp\cO[x_{ij}^{(1)},\dots,x_{ij}^{(m)}]$,
the coordinate ring of $X_{\lambda}$ is torsion free, hence flat
over $\cO$). By \cite[II~2.1]{SGA1} it therefore suffices to prove
that the fibres are smooth. 

Let $K$ be either $F$ or $\F_{q}$ and let $f((x_{ij}^{(1)}),\dots,(x_{ij}^{(m)}))=\prod_{\ell=1}^{m}\det(x_{ij}^{(\ell)})^{e_{\ell}}-1$.
By the Jacobian criterion, $X_{\lambda}\times K$ is smooth if the
gradient $\nabla f$ is non-zero at every point of $X_{\lambda}(K)$.
Without loss of generality, we may assume that $p\nmid e_{1}$.

Suppose that $\bfa=(a_{1},\dots,a_{m})\in X_{\lambda}(K)$ is a point
such that $\nabla f(\bfa)=0$. For any $1\leq u,v\leq d_{1}$, we
have, by the chain rule, 
\[
\frac{\partial f((x_{ij}^{(1)}),\dots,(x_{ij}^{(m)}))}{\partial x_{uv}^{(1)}}=e_{1}\det(x_{ij}^{(1)})^{e_{1}-1}\frac{\partial\det(x_{ij}^{(1)})}{\partial x_{uv}^{(1)}}\prod_{\ell=2}^{m}\det(x_{ij}^{(\ell)})^{e_{\ell}}.
\]
Evaluating at the point $\bfa$, we thus obtain
\[
e_{1}\det(a_{1})^{e_{1}-1}\frac{\partial\det(x_{ij}^{(1)})}{\partial x_{uv}^{(1)}}(a_{1})\prod_{\ell=2}^{m}\det(a_{\ell})^{e_{\ell}}=0.
\]
Since $\chara K$ is either $p$ or $0$ and $p\nmid e_{1}$, we can
cancel the factor $e_{1}$ in the above equation. Moreover, since
$\det(a_{i})\neq0$ for every component $a_{i}$ of $\bfa$, we can
cancel the factors $\det(a_{1})^{e_{1}-1}$ and $\prod_{\ell=2}^{m}\det(a_{\ell})^{e_{\ell}}$.
We are thus left with
\[
\frac{\partial\det(x_{ij}^{(1)})}{\partial x_{uv}^{(1)}}(a_{1})=0,
\]
for any $1\leq u,v\leq d_{1}$. Now, the numbers $\frac{\partial\det(x_{ij}^{(1)})}{\partial x_{uv}^{(1)}}(a_{1})$
are the entries of the gradient of $\det(x_{ij}^{(1)})$ at $a_{1}$.
But by the smoothness of $\GL_{N}$ over $K$ for any $N$, this implies
that $\det(a_{1})=0$ (the gradient of $\det(x_{ij}^{(1)})$ is non-zero
at every point of $\GL_{N}(K)$). This contradicts the fact that $\bfa\in X_{\lambda}(K)$,
because $\prod_{\ell=1}^{m}\det(a_{\ell})^{e_{\ell}}=1$. Thus, we
have proved that 
\[
\nabla f(\bfa)\neq0
\]
 for every point $\bfa\in X_{\lambda}(K)$, so $X_{\lambda}\times K$
is smooth over $K$, and hence $X_{\lambda}$ is smooth over $\cO$.
\end{proof}
\begin{lem}
\label{lem:X_lambda-surjective}For every $\lambda=(d_{1}^{e_{1}},\dots,d_{m}^{e_{m}})$,
the map
\[
X_{\lambda}(\cO)\longrightarrow X_{\lambda}(\F_{q})
\]
 is surjective.
\end{lem}

\begin{proof}
Let $m$ be the largest integer such that $p^{m}\mid e_{i}$, for
all $i\in\{1,\dots,m\}$, and write $e_{i}=p^{m}e_{i}'$ for integers
$e_{i}'$. Let $\mu=(d_{1}^{e_{1}'},\dots,d_{m}^{e_{m}'})$. By Lemma~\ref{lem:smooth-X_lambda}
$X_{\mu}$ is smooth over $\cO$, so by the infinitesimal criterion
for smoothness, the map
\[
X_{\mu}(\cO)\longrightarrow X_{\mu}(\F_{q})
\]
 is surjective. The scheme $X_{\mu}\times\F_{q}$ is defined by the
equation $\prod_{\ell=1}^{m}\det(x_{ij}^{(\ell)})^{e_{\ell}'}=1$
and $X_{\lambda}\times\F_{q}$ is defined by the equation
\[
\prod_{\ell=1}^{m}\det(x_{ij}^{(\ell)})^{e_{\ell}}-1=\left(\prod_{\ell=1}^{m}\det(x_{ij}^{(\ell)})^{e_{\ell}'}-1\right)^{p^{m}}.
\]
Hence $X_{\mu}(\F_{q})=X_{\lambda}(\F_{q})$ and so $X_{\mu}(\cO)$
maps surjectively onto $X_{\lambda}(\F_{q})$. But every solution
to the equation $\prod_{\ell=1}^{m}\det(x_{ij}^{(\ell)})^{e_{\ell}'}=1$
is also a solution to the equation $\prod_{\ell=1}^{m}\det(x_{ij}^{(\ell)})^{e_{\ell}}=\left(\prod_{\ell=1}^{m}\det(x_{ij}^{(\ell)})^{e_{\ell}'}\right)^{p^{m}}=1$,
so $X_{\mu}(\cO)\subseteq X_{\lambda}(\cO)$, and thus $X_{\lambda}(\cO)$
maps surjectively onto $X_{\lambda}(\cO)$.
\end{proof}

\subsection{Proof of Theorem~\ref{thm:Thm1}}

We can now finish the proof of Theorem~\ref{thm:Thm1}. Let $x\in\M_{n}(\F_{q})$
be in Weyr form, let $n_{1},\dots,n_{r}$ be the sizes of the Weyr
blocks, and let $\lambda_{i}=(n_{1}^{(i)},,\dots,n_{m_{i}}^{(i)})$,
$i=1,\dots,r$ be the partition of $n_{i}$ determined by the $i$-th
Weyr block. For any commutative ring $A$, let $N_{i}(A)$ be as in
Lemma~\ref{lem:Weyr-centraliser}, corresponding to the $i$-th Weyr
block.

By Lemmas~\ref{lem:C-decomp-in-blocks} and \ref{lem:Weyr-centraliser},
we can write every element in $C_{S}(x)$ as the sum, taken in $\M_{n}(\F_{q})$,
of an element in 
\[
\SL_{n}(\F_{q})\cap\prod_{i=1}^{r}\prod_{\ell=1}^{m_{i}}\M_{n_{\ell}^{(i)}}(\F_{q})
\]
and an element in $\bigoplus_{i=1}^{r}N_{i}(\F_{q})$. 

Similarly, since $s(x)\in\M_{n}(\cO_{2})$ is in Weyr form (with the
same block sizes), every element in $C_{S_{2}}(s(x))$ is the sum,
taken in $\M_{n}(\cO_{2})$, of an element in 
\[
\SL_{n}(\cO_{2})\cap\prod_{i=1}^{r}\prod_{\ell=1}^{m_{i}}\M_{n_{\ell}^{(i)}}(\cO_{2})
\]
 and an element in $\bigoplus_{i=1}^{r}N_{i}(\cO_{2})$. The map $\rho:\bigoplus_{i=1}^{r}N_{i}(\cO_{2})\rightarrow\bigoplus_{i=1}^{r}N_{i}(\F_{q})$
is surjective because $\rho:E_{ij}(\cO_{2})\rightarrow E_{ij}(\F_{q})$
is surjective for each $1\leq i,j,\leq n$. 

As subschemes of $\SL_{n}$ over $\cO$, we have 
\[
\SL_{n}\cap\prod_{i=1}^{r}\prod_{\ell=1}^{m_{i}}\M_{n_{\ell}^{(i)}}=X_{\lambda},
\]
with $\lambda=(n_{1}^{(1)},,\dots,n_{m_{1}}^{(1)},\dots,n_{1}^{(r)},,\dots,n_{m_{r}}^{(r)})$,
so  Lemma~\ref{lem:X_lambda-surjective} implies that the map
\[
\rho:\SL_{n}(\cO_{2})\cap\prod_{i=1}^{r}\prod_{\ell=1}^{m_{i}}\M_{n_{\ell}^{(i)}}(\cO_{2})\longrightarrow\SL_{n}(\F_{q})\cap\prod_{i=1}^{r}\prod_{\ell=1}^{m_{i}}\M_{n_{\ell}^{(i)}}(\F_{q})
\]
is surjective. Thus $\rho:C_{S_{2}}(s(x))\rightarrow C_{S}(x)$ is
surjective.

\section{\label{sec:Structure-of-the}Structure of the stabiliser}

The adjoint action of $G$ on $\M_{n}(\F_{q})$ induces an action
of $G$ on $\M_{n}(\F_{q})/Z$, where $Z$ is the scalar matrices
as in Section~\ref{sec:Notational-preliminaries}. Let $C_{G}(x+Z)$
denote the centraliser of $x+Z\in\M_{n}(\F_{q})/Z$ and $C_{S}(x+Z)=C_{G}(x+Z)\cap S$.
Letting $S_{2}$ act on $\M_{n}(\F_{q})/Z$ via $S$, we have the
centraliser $C_{S_{2}}(x+Z)=\rho^{-1}(C_{S}(x+Z))$. The definition
of the character $\psi_{x+Z}$, $x\in\M_{n}(\F_{q})$ of $S^{1}$
(see Section~\ref{sec:Notational-preliminaries}) implies that
\[
\Stab_{S_{2}}(\psi_{x+Z})=C_{S_{2}}(x+Z).
\]

Our goal in Section~\ref{sec:The-main-result} is to prove that for
any $x\in\M_{n}(\F_{q})$, the character $\psi_{x+Z}$ extends to
its stabiliser in $S_{2}$. In the present section, we will therefore
study the structure of $C_{S}(x+Z)$ and $C_{S_{2}}(x+Z)$. A first
easy observation is that $C_{S}(x)$ is a normal subgroup of $C_{S}(x+Z)$.
\begin{lem}
\label{lem:V-complement}Assume that $x\in\M_{n}(\F_{q})$ is in Weyr
normal form. Then there exists a permutation matrix $v\in S$ which
is either the identity or of order $p$ (where $p=\chara\F_{q}$),
such that
\[
C_{S}(x+Z)=\langle v\rangle C_{S}(x).
\]
\end{lem}

\begin{proof}
Write 
\[
x=W(a_{1})\oplus\dots\oplus W(a_{r}),
\]
where $W(a_{i})$ is a Weyr block of $x$ with eigenvalue $a_{i}$
and size $n_{i}$. We thus have $a_{i}\neq a_{j}$ for all $1\leq i<j\leq r$. 

Let $g\in C_{S}(x+Z)$, so that $gxg^{-1}=x+\lambda I$ for some $\lambda\in\F_{q}$.
If $\lambda=0$, we have $g\in C_{S}(x)$. Assume now that $\lambda\neq0$.
The sets of eigenvalues of $x$ and $x+\lambda I$ agree:
\[
\{a_{1},\dots,a_{r}\}=\{a_{1}+\lambda,\dots,a_{r}+\lambda\},
\]
so we have $r\lambda=0$, hence $p\mid r$ . Moreover, we have a permutation
$\sigma$ in the symmetric group $\cS_{r}$ defined by $a_{i}\mapsto a_{i}+\lambda=:a_{\sigma(i)}$,
and
\[
x+\lambda I=W(a_{\sigma^{-1}(1)})\oplus\dots\oplus W(a_{\sigma^{-1}(r)})
\]
is in Weyr form, hence $n_{i}=n_{\sigma^{-1}(i)}$, so all the blocks
$W(a_{i})$ where $i$ is in an orbit of $\sigma$ are of the same
size. 

For every $i=1,\dots,r$, we have
\[
a_{i}=a_{i}+p\lambda=a_{\sigma^{p}(i)},
\]
and thus $\sigma^{p}=\Id$. We conclude that in particular 
\[
\{a_{1},a_{\sigma(1)},\dots,a_{\sigma^{p-1}(1)}\}=\{a_{1}+\lambda,a_{1}+2\lambda,\dots,a_{1}+(p-1)\lambda\},
\]
so $\lambda=a_{\sigma^{i}(1)}-a_{1}$, for some $i\in\{1,\dots,p-1\}$,
and there are therefore at most $p-1$ distinct possibilities for
$\lambda$ when $\lambda\neq0$. Since $g^{i}xg^{-i}=x+i\lambda I$,
for $i\in\N$, there are exactly $p-1$ distinct possible nonzero
values of $\lambda$, namely 
\[
\lambda,2\lambda,\dots,(p-1)\lambda.
\]
Thus, for any $h\in C_{S}(x+Z)$, there exists an $i\in\N\cup\{0\}$
such that
\[
hxh^{-1}=x+i\lambda I=g^{i}xg^{-i},
\]
and therefore $h\in g^{i}C_{S}(x)$.

Let $v\in G=\GL_{n}(\F_{q})$ be the permutation matrix determined
by the permutation $\sigma$ of the Weyr blocks $W(a_{i})$. More
precisely, let $v_{0}\in\GL_{r}(\F_{q})$ be the (column-wise) permutation
matrix defined by $\sigma\in\cS_{r}$, and let $v\in G$ be the matrix,
blocked according to the partition $n=n_{1}+\dots+n_{r}$, obtained
from $v_{0}$ by replacing the $1$-entry in column $i$ in $v_{0}$
by an $n_{i}\times n_{i}$ identity matrix. Then $v$ is a permutation
matrix of order $p$ such that
\[
vxv^{-1}=x+\lambda I.
\]
If $p=2$ then $\det(v)=1$, so $v\in S$. If $p\neq2$ then both
$v$ and $v^{2}$ have order $p$, so $v=v^{2i}$, for some $i\in\N$,
and hence $\det(v)=\det(v)^{2i}=(-1)^{2i}=1$, so $v\in S$. 

We have thus proved that if $C_{S}(x+Z)=C_{S}(x)$, then we can take
$v=I\in\M_{n}(\F_{q})$, and, otherwise, if there exists a $g\in C_{S}(x+Z)$,
such that $gxg^{-1}=x+\lambda I$ for some $\lambda\neq0$, then $gC_{S}(x)=v^{i}C_{S}(x)$,
for some $i\in\{1,\dots,p-1\}$, so that
\[
C_{S}(x+Z)=\langle v\rangle C_{S}(x).
\]
\end{proof}
The first two parts of the following lemma are partially contained
in the proof in \cite[Section~2.4]{Pooja-classicalgrps}. Note however
that a set of representatives of $C_{S_{2}}(x+Z)/C_{S_{2}}(x)$ cannot
in general consist only of permutation matrices when $p=\chara\F_{q}=2$.
For example, for $p=2$, the matrix $\left(\begin{smallmatrix}0 & 1\\
1 & 0
\end{smallmatrix}\right)$ is not an element in $\SL_{2}(W_{2}(\F_{q}))$, although its image
in $\M_{2}(\F_{q})$ is in $\SL_{2}(\F_{q})$. For this reason, the
lemma is not claiming that $w$ is a permutation matrix, and the possible
non-triviality of $u$ in the third condition is a precise way to
account for this fact.
\begin{lem}
\label{lem:V-complement-lifted}Assume that $x\in\M_{n}(\F_{q})$
is in Weyr normal form and let $v\in S$ be as in Lemma~\ref{lem:V-complement}.
Then there exists a matrix $w\in S_{2}$ (possibly equal to the identity)
such that the following conditions hold:
\begin{enumerate}
\item $\rho(w)=v$,
\item $C_{S_{2}}(x+Z)=\langle w\rangle C_{S_{2}}(x)$,
\item \label{enu:wcw-ucu}Write $x=W(a_{1})\oplus\dots\oplus W(a_{r})$,
where $W(a_{i})$ is a Weyr block of size $n_{i}$ with eigenvalue
$a_{i}$ and $a_{i}\neq a_{j}$ for all $1\leq i<j\leq r$. Let $\sigma\in\cS_{r}$
be the permutation such that $a_{\sigma(i)}=a_{i}+\lambda I$ and
$vxv^{-1}=x+\lambda I$, for some $\lambda\in k$. Let $c\in C_{S_{2}}(s(x))$
with $c=c_{1}\oplus\dots\oplus c_{r}$ for $c_{i}\in\GL_{n_{i}}(\cO_{2})$.
Then there exists a $u\in\GL_{N}(\cO_{2})$, with $N=n_{\sigma^{-1}(1)}$,
such that 
\[
wcw^{-1}=uc_{\sigma^{-1}(1)}u^{-1}\oplus\dots\oplus c_{\sigma^{-1}(r)}.
\]
\end{enumerate}
\end{lem}

\begin{proof}
Let $v\in S$ be as in Lemma~\ref{lem:V-complement}, so that $C_{S}(x+Z)=\langle v\rangle C_{S}(x)$.
Let $\hat{v}$ be the same permutation matrix viewed as an element
in $G_{2}$. If $p$ is odd, we must have $\hat{v}\in S_{2}$ since
$\rho(\det(\hat{v}))=\det(\rho(\hat{v}))=1$ implies that $\det(\hat{v})\neq-1$,
and hence that $\det(\hat{v})=1$. When $p=2$, $\rho(\det(\hat{v}))=\det(v)=1$
does not imply that $\det(\hat{v})\neq-1$, so we will modify $\hat{v}$,
if necessary. Define
\[
w=\begin{cases}
\hat{v} & \text{if }\det(\hat{v})=1,\\
\diag(-1,1,\dots,1)\hat{v} & \text{if }p=2\text{ and }\det(\hat{v})\neq1,
\end{cases}
\]
so that in either case we have $w\in S_{2}$ and $\rho(w)=v$, and
thus $C_{S_{2}}(x+Z)=\langle w\rangle C_{S_{2}}(x)$.

It remains to prove the final assertion. By the proof of Lemma~\ref{lem:V-complement}
and the choice of $\hat{v}$, we have $\hat{v}c\hat{v}^{-1}=c_{\sigma^{-1}(1)}\oplus\dots\oplus c_{\sigma^{-1}(r)}$.
Thus, when $\det(\hat{v})=1$ we can take $u=I\in\GL_{N}(\cO_{2})$.
Assume now that $p=2\text{ and }\det(\hat{v})\neq1$. Then
\begin{align*}
wcw^{-1} & =\diag(-1,1,\dots,1)(c_{\sigma^{-1}(1)}\oplus\dots\oplus c_{\sigma^{-1}(r)})\diag(-1,1,\dots,1)^{-1}\\
 & =uc_{\sigma^{-1}(1)}u^{-1}\oplus\dots\oplus c_{\sigma^{-1}(r)},
\end{align*}
where $u=\diag(-1,1,\dots,1)\in\GL_{N}(\cO_{2})$.
\end{proof}
The proof of the above lemma shows that $w$ can be taken to be a
signed permutation matrix.

\section{\label{sec:The-main-result}The main result}

We will now use the results established in the previous two sections
to prove our main result. Since the proof involves passing to finite
extension $k$ of the ground field $\F_{q}$, we start by setting
up some notation and note a few immediate consequences.

For every $x\in\M_{n}(\F_{q})$, there exists a finite field extension
$k/\F_{q}$ such that the Weyr form of $x$ lies in $\M_{n}(k)$.
Indeed, we may construct $k$ by adjoining all the eigenvalues of
$x$ to $\F_{q}$. Then there exists a $g\in\GL_{n}(k)$ such that
the Weyr form of $x$ is $gxg^{-1}$ (this is a well-known consequence
of the existence and uniqueness of rational canonical forms over a
field).

Let $k/\F_{q}$ be a finite field extension and let $\tilde{\cO}$
be an unramified extension of $\cO$ with residue field $k$. Let
\[
\tilde{G}_{2}=\GL_{n}(\tilde{\cO}_{2}),\qquad\tilde{G}=\GL_{n}(k),
\]
and define $\tilde{S}_{2},\tilde{S},\tilde{G}^{1},\tilde{S}^{1}$
analogously. Furthermore, let $\tilde{Z}$ denote the subalgebra of
$\M_{n}(k)$ of scalar matrices.

We use the notation $\rho$ (same as over $\F_{q}$) for the map $\rho:\M_{n}(\tilde{\cO}_{2})\rightarrow\M_{n}(k)$
and its restrictions to homomorphisms of $\tilde{G}_{2}$ and $\tilde{S}_{2}$.
We also use the notation $s:\M_{n}(k)\rightarrow\M_{n}(\tilde{\cO}_{2})$
(same as over $\F_{q}$) for the function induced by the multiplicative
section $k\rightarrow\tilde{\cO}_{2}$.

Since Theorem~\ref{thm:Thm1} and Lemma~\ref{lem:V-complement-lifted}
hold for $\cO_{2}$ with an arbitrary finite residue field $\F_{q}$,
they also hold with $\F_{q}$, $\cO_{2}$, $S_{2}$ and $S$ replaced
by $k$, $\tilde{\cO}_{2}$, $\tilde{S}_{2}$ and $\tilde{S}$, respectively.
Moreover, the proofs are the same, up to changing the notation accordingly.
We record this formally:
\begin{cor}
\label{cor:Thm1 and lem-V-complement-lift}Let $k/\F_{q}$ be a finite
field extension and let $\tilde{\cO}$ be an unramified extension
of $\cO$ with residue field $k$. Then Theorem~\ref{thm:Thm1},
Lemma~\ref{lem:C-decomp-in-blocks} and Lemma~\ref{lem:V-complement-lifted}
hold with $\F_{q}$, $\cO_{2}$, $S_{2}$, $S$ and $Z$ replaced
by $k$, $\tilde{\cO}_{2}$, $\tilde{S}_{2}$, $\tilde{S}$ and $\tilde{Z}$,
respectively.
\end{cor}

From now on, let $x\in\M_{n}(\F_{q})$. We have a character $\psi_{x+Z}$
of $S^{1}$, as well as a character $\psi_{x+\tilde{Z}}$ of $\tilde{S}^{1}$
(see Section~\ref{sec:Notational-preliminaries}). It follows immediately
from the definitions of these characters that $\psi_{x+\tilde{Z}}$
is an extension of $\psi_{x+Z}$.

As noted in the beginning of Section~\ref{sec:Structure-of-the},
we have $\Stab_{S_{2}}(\psi_{x+Z})=C_{S_{2}}(x+Z)$, and similarly,
$\Stab_{\tilde{S}_{2}}(\psi_{x+\tilde{Z}})=C_{\tilde{S}_{2}}(x+\tilde{Z})$.
Since $Z\subseteq\tilde{Z}$, we obviously have $C_{S}(x+Z)\subseteq C_{S}(x+\tilde{Z})$.
It follows that $C_{S_{2}}(x+Z)$ is a subgroup of $C_{S_{2}}(x+\tilde{Z})$,
hence of $C_{\tilde{S}_{2}}(x+\tilde{Z})$.

The following result is the key to proving our main result.
\begin{prop}
\label{prop:extension-exists}For any $x\in\M_{n}(\F_{q})$, the character
$\psi_{x+Z}$ extends to $C_{S_{2}}(x+Z)$.
\end{prop}

\begin{proof}
Let $k/\F_{q}$ and $g\in\tilde{G}$ be such that $y:=gxg^{-1}\in\M_{n}(k)$
is in Weyr form. By Corollary~\ref{cor:Thm1 and lem-V-complement-lift},
$\rho:C_{\tilde{S}_{2}}(s(y))\rightarrow C_{\tilde{S}}(y)$ is surjective,
so $C_{\tilde{S}_{2}}(y)=C_{\tilde{S}_{2}}(s(y))\tilde{S}^{1}$. By
Corollary~\ref{cor:Thm1 and lem-V-complement-lift} Lemma~\ref{lem:V-complement-lifted}
holds for $\tilde{\cO}_{2}$ with residue field $k$. Thus let $w\in\tilde{S}_{2}$
be as in Lemma~\ref{lem:V-complement-lifted}, with respect to the
element $y$ (instead of $x$), so that
\begin{equation}
C_{\tilde{S}_{2}}(y+\tilde{Z})=\langle w\rangle C_{\tilde{S}_{2}}(y)=\langle w\rangle C_{\tilde{S}_{2}}(s(y))\tilde{S}^{1}.\label{eq:N-is-w-C-S1}
\end{equation}
Let $\hat{g}\in\tilde{G}_{2}$ be a lift of $g$. Then $C_{\tilde{S}_{2}}(y+\tilde{Z})=\hat{g}C_{\tilde{S}_{2}}(x+\tilde{Z})\hat{g}^{-1}$,
as this relation follows from the corresponding relation mod $\tilde{S}^{1}$.

To prove that $\psi_{x+Z}$ extends to $C_{S_{2}}(x+Z)$ we claim
that it is enough to show that the character $\leftexp{\hat{g}^{-1}}{\psi_{x+\tilde{Z}}}=\psi_{gxg^{-1}+\tilde{Z}}=\psi_{y+\tilde{Z}}$
of $\tilde{S}^{1}$ has an extension to $C_{\tilde{S}_{2}}(s(y))\tilde{S}^{1}$
which is fixed by $w$. Indeed, by a well-known result in Clifford
theory, an irreducible representation of a normal subgroup $N$ of
a finite group $G$ extends to $G$ if it is fixed by $G$ (see \cite[( 11.22)]{Isaacs}).
Thus, if $\psi_{y+\tilde{Z}}$ $\tilde{S}^{1}$ has an extension to
$C_{\tilde{S}_{2}}(s(y))\tilde{S}^{1}$ which is fixed by $w$, then
(\ref{eq:N-is-w-C-S1}) implies that $\psi_{y+\tilde{Z}}$ extends
to $C_{\tilde{S}_{2}}(y+\tilde{Z})$ and hence that $\leftexp{\hat{g}}{\psi_{y+\tilde{Z}}}=\psi_{x+\tilde{Z}}$
extends to $\hat{g}^{-1}C_{\tilde{S}_{2}}(y+\tilde{Z})\hat{g}=C_{\tilde{S}_{2}}(x+\tilde{Z})$.
Finally, restricting this extension of $\psi_{x+\tilde{Z}}$ to the
subgroup $C_{S_{2}}(x+Z)$, we obtain a degree one character which
contains $\psi_{x+Z}$.

We now proceed to construct an explicit extension of $\psi_{y+\tilde{Z}}$
to $C_{\tilde{S}_{2}}(s(y))\tilde{S}^{1}$ and show that it is fixed
by $w$. Write 
\[
y=W(a_{1})\oplus\dots\oplus W(a_{r}),
\]
where each $W(a_{i})$ is a Weyr block of size $n_{i}$ with eigenvalue
$a_{i}$, and $a_{i}\neq a_{j}$ for all $1\leq i<j\leq r$. By Corollary~\ref{cor:Thm1 and lem-V-complement-lift}
Lemma~\ref{lem:C-decomp-in-blocks} holds for $\tilde{\cO}_{2}$
with residue field $k$. We therefore have
\[
C_{\tilde{G}_{2}}(s(y))=C_{\tilde{G}_{2}}(s(W(a_{1})))\times\dots\times C_{\tilde{G}_{2}}(s(W(a_{r}))),
\]
so for any $c\in C_{\tilde{S}_{2}}(s(y))$ we can write $c=c_{1}\oplus\dots\oplus c_{r}$
with $c_{i}\in C_{\tilde{G}_{2}}(s(W(a_{i})))$ and $\det(c_{1})\cdots\det(c_{r})=1$.

For each $i=1,\dots,r$, let $\chi_{a_{i}}\in\Irr(\tilde{\cO}_{2}^{\times})$
be the character such that 
\begin{align*}
\chi_{a_{i}}(s(k^{\times})) & =1,\quad\text{and}\\
\chi_{a_{i}}(1+\varpi\beta) & =\psi(a_{i}\rho(\beta)),\quad\text{for }\beta\in\tilde{\cO}_{2}.
\end{align*}
Note that $\chi_{a_{i}}$ is well-defined since $\tilde{\cO}_{2}^{\times}\cong s(k^{\times})\times(1+\varpi\tilde{\cO}_{2})$
and $1+\varpi\beta\mapsto\rho(\beta)$ is an isomorphism $1+\varpi\tilde{\cO}_{2}\cong k$.

Define a degree one character $\chi$ of $C_{\tilde{S}_{2}}(s(y))$
by 
\[
\chi(c_{1}\oplus\dots\oplus c_{r})=\prod_{i=1}^{r}\chi_{a_{i}}(\det(c_{i})),
\]
for any $c=c_{1}\oplus\dots\oplus c_{r}\in C_{\tilde{S}_{2}}(s(y))$.
We will now show that $\chi$ agrees with $\psi_{y+\tilde{Z}}$ on
the intersection $C_{\tilde{S}_{2}}(s(y))\cap\tilde{S}^{1}$.

For any $c=c_{1}\oplus\dots\oplus c_{r}\in C_{\M_{n}(\tilde{\cO}_{2})}(s(y))$
we have $\bar{c}=\bar{c}_{1}\oplus\dots\oplus\bar{c}_{r}\in C_{\M_{n}(k)}(y)$,
where $\bar{c}=\rho(c)$ and $\bar{c}_{i}=\rho(c_{i})$. From the
explicit description of $C_{\M_{n}(\tilde{\cO}_{2})}(W(a_{i}))$ in
Lemma~\ref{lem:Weyr-centraliser}, one sees by direct computation
with block matrices that for any $\bar{c}_{i}\in C_{\M_{n}(\tilde{\cO}_{2})}(W(a_{i}))$,
we have
\[
\Tr(W(a_{i})\bar{c}_{i})=a_{i}\Tr(\bar{c}_{i}).
\]
Now let $1+\varpi c\in C_{\tilde{S}_{2}}(s(y))\cap\tilde{S}^{1}$
where $c=c_{1}\oplus\dots\oplus c_{r}\in C_{\M_{n}(\tilde{\cO}_{2})}(s(y))$.
Then, 
\begin{align*}
\psi_{y+\tilde{Z}}(1+\varpi c) & =\psi(\Tr(y\bar{c}))=\psi\Big(\Tr\Big(\sum_{i=1}^{r}W(a_{i})\bar{c}_{i}\Big)\Big)\\
 & =\prod_{i=1}^{r}\psi(a_{i}\Tr(\bar{c}_{i}))=\prod_{i=1}^{r}\chi_{a_{i}}(\det(1+\varpi\bar{c}_{i})),
\end{align*}
and thus $\chi$ equals $\psi_{y+\tilde{Z}}$ on $C_{\tilde{S}_{2}}(s(y))\cap\tilde{S}^{1}$.
We can therefore glue $\chi$ and $\psi_{y+\tilde{Z}}$ to a degree
one character $\chi\psi_{y+\tilde{Z}}$ of $C_{\tilde{S}_{2}}(s(y))\tilde{S}^{1}$,
which is an extension of $\psi_{y+\tilde{Z}}$. 

It remains to show that $\chi\psi_{y+\tilde{Z}}$ is fixed by $w$.
Since $w\in C_{\tilde{S}_{2}}(y+\tilde{Z})=\Stab_{\tilde{S}_{2}}(\psi_{y+\tilde{Z}})$,
it is enough to show that $w$ fixes $\chi$. If $w=I$ there is nothing
to prove, so assume that $w\neq I$. Let $c=c_{1}\oplus\dots\oplus c_{r}\in C_{\tilde{S}_{2}}(s(y))$
and write $\det(c_{i})=\alpha_{i}(1+\varpi\hat{\beta}_{i})$, where
$\alpha_{i}\in s(k^{\times})$ and $\beta_{i}=\rho(\hat{\beta}_{i})\in k$
are uniquely determined by $c_{i}$. Then 
\[
1=\det(c)=\big(\prod_{i=1}^{r}\alpha_{i}\big)\big(1+\varpi\sum_{i=1}^{r}\hat{\beta}_{i}\big),
\]
implying that $\sum_{i=1}^{r}\beta_{i}=0$.

By Lemma~\ref{lem:V-complement-lifted}\,\ref{enu:wcw-ucu}, we
have 
\[
wcw^{-1}=uc_{\sigma^{-1}(1)}u^{-1}\oplus\dots\oplus c_{\sigma^{-1}(r)},
\]
for some $u\in\GL_{N}(\tilde{\cO}_{2})$, where $N$ is the size of
$c_{\sigma^{-1}(1)}$, and $\sigma\in\cS_{r}$ is a permutation such
that there is a non-zero $\lambda\in k$ such that for all $i$, we
have $a_{\sigma(i)}=a_{i}+\lambda I$. Thus
\begin{align*}
\chi(wcw^{-1}) & =\chi_{a_{1}}(\det(uc_{\sigma^{-1}(1)}u^{-1}))\prod_{i=2}^{r}\chi_{a_{i}}(\det(c_{\sigma^{-1}(i)}))=\prod_{i=1}^{r}\chi_{a_{i}}(\det(c_{\sigma^{-1}(i)}))\\
 & =\prod_{i=1}^{r}\psi(a_{i}\beta_{\sigma^{-1}(i)})=\psi\big(\sum_{i=1}^{r}a_{i}\beta_{\sigma^{-1}(i)}\big)=\psi\big(\sum_{i=1}^{r}(a_{\sigma^{-1}(i)}+\lambda I)\beta_{\sigma^{-1}(i)}\big)\\
 & =\psi\big(\sum_{i=1}^{r}a_{\sigma^{-1}(i)}\beta_{\sigma^{-1}(i)}\big)=\psi\big(\sum_{i=1}^{r}a_{i}\beta_{i}\big)=\prod_{i=1}^{r}\chi_{a_{i}}(\det(c_{i}))=\chi(c),
\end{align*}
where we have used that $\sum_{i=1}^{r}\beta_{\sigma^{-1}(i)}=\sum_{i=1}^{r}\beta_{i}=0$.

We have thus shown that $\chi\psi_{x+\tilde{Z}}$ is fixed by $w$,
hence it is fixed by the group $\langle w\rangle$. By the reduction
steps given above, this finishes the proof.
\end{proof}
We can now deduce our main result.
\begin{thm}
\label{thm:Main}For all integers $n\geq1$ and $d\geq1$, we have
\[
\#\Irr_{d}(\SL_{n}(\F_{q}[t]/t^{2}))=\#\Irr_{d}(\SL_{n}(W_{2}(\F_{q}))).
\]
\end{thm}

\begin{proof}
By Proposition~\ref{prop:extension-exists} the character $\psi_{x+Z}$
extends to $C_{S_{2}}(x+Z)$, for every $x\in\M_{n}(\F_{q})$. Thus,
by well known results in Clifford theory \cite[6.11, 6.17]{Isaacs},
there is a bijection
\begin{align*}
\Irr(C_{S_{2}}(x+Z)/S^{1}) & \longrightarrow\Irr(S_{2}\mid\psi_{x+Z})\\
\theta & \longmapsto\pi(\theta):=\Ind_{C_{S_{2}}(x+Z)}^{S_{2}}(\theta\tilde{\psi}_{x+Z}).
\end{align*}
Thus $\#\Irr(S_{2}\mid\psi_{x+Z})=\#\Irr(C_{S_{2}}(x+Z)/S^{1})=\#\Irr(C_{S}(x+Z))$
and $\dim\pi(\theta)=[S_{2}:C_{S_{2}}(x+Z)]\cdot\dim\theta=[S:C_{S}(x+Z)]\cdot\dim\theta$,
so $\#\Irr_{d}(S_{2})$ only depends on $|S|$ and $|C_{S}(x+Z)|$,
where $x+Z\in\M_{n}(\F_{q})/Z$ runs through a set of representatives
of the orbits under the conjugation action by $S$.
\end{proof}
\begin{rem}
In the proof of Proposition~\ref{prop:extension-exists}, the construction
of $\chi$ and the argument for showing that it is fixed by $w$ is
due to Singla in \cite[Lemma~2.3]{Pooja-classicalgrps} in the case
where $w$ is a permutation matrix. As noted before Lemma~\ref{lem:V-complement-lifted},
the latter is not always the case.
\end{rem}

\section{\label{sec:The-isomorphism-problem}The isomorphism problem}

In this final section we prove that for $p\geq5$ the groups $\SL_{n}(W_{2}(\F_{q}))$
and $\SL_{n}(\F_{q}[t]/t^{2})$ are never isomorphic. As is easy to
see (and as we will show below), this follows if we can show that
the group extension
\begin{equation}
1\longrightarrow K\longrightarrow\SL_{n}(W_{2}(\F_{q}))\longrightarrow\SL_{n}(\F_{q})\longrightarrow1,\label{eq:group-extension}
\end{equation}
where $K$ denotes the kernel of $\rho:\SL_{n}(W_{2}(\F_{q}))\rightarrow\SL_{n}(\F_{q})$,
is not split. This non-splitting holds if and only if $(n,p)\not\in\{(2,3),(3,2)\}$
(see Theorem~\ref{thm:non-splitting} below). For $q=p$ and $p\geq5$,
this follows from a result of Serre \cite[Ch.~IV, 3.4, Lemma~3]{Serre-Abelian-l-adic}
(the proof is given for $\SL_{2}$ but generalises to $\SL_{n}$ for
$n>2$). For $q=p$ it was proved by Sah \cite[Theorem~7]{Sah-I}
and for general $q$ and $\GL_{n}$, an argument was sketched in \cite[Proposition~0.3]{Sah-II}.
In the case $p\geq5$, we give a new simple proof. We note that \cite[Theorem~1.2]{Kondratiev-Zalesskii}
is misstated, as in fact $\SL_{2}(\Z/4)\not\cong\SL_{2}(\F_{2}[t]/t^{2})$
(see Remark~\ref{rem:Theorem-also-holds}).

A finite group $G$ is said to split over a normal subgroup $N$ if
the exact sequence $1\rightarrow N\rightarrow G\rightarrow G/N\rightarrow1$
splits, that is, if there exists a subgroup $Q$ of $G$ such that
$G=QN$ and $Q\cap N=1$.

Let $L$ denote the kernel of $\rho:\SL_{n}(\F_{q}[t]/t^{2})\rightarrow\SL_{n}(\F_{q})$.
Then $K\cong L\cong\sl_{n}(\F_{q})$, the additive group of the ring
of trace zero matrices over $\F_{q}$. The first thing to note is
that $\SL_{n}(\F_{q}[t]/t^{2})$ is split over the kernel of $L$
because the homomorphism $\SL_{n}(\F_{q})\rightarrow\SL_{n}(\F_{q}[t]/t^{2})$
induced by the ring homomorphism $\F_{q}\rightarrow\F_{q}[t]/t^{2}$,
$x\mapsto x+yt$ is a section of $\rho$. 

Moreover, $K$ is the maximal normal $p$-subgroup of $\SL_{n}(W_{2}(\F_{q}))$
and $L$ is the maximal normal $p$-subgroup of $\SL_{n}(\F_{q}[t]/t^{2})$,
so for any isomorphism 
\[
\alpha:\SL_{n}(W_{2}(\F_{q}))\rightarrow\SL_{n}(\F_{q}[t]/t^{2}),
\]
we have $\alpha(K)=L$, and thus, if $\SL_{n}(W_{2}(\F_{q}))\cong\SL_{n}(\F_{q}[t]/t^{2})$,
then (\ref{eq:group-extension}) must split (indeed, since $\SL_{n}(\F_{q}[t]/t^{2})=LQ$,
for some $Q$ such that $Q\cap L=1$, we get $\SL_{n}(W_{2}(\F_{q}))=K\alpha^{-1}(Q)$).
Therefore, to prove that $\SL_{n}(W_{2}(\F_{q}))$ and $\SL_{n}(\F_{q}[t]/t^{2})$
are not isomorphic, it suffices to prove that $\SL_{n}(W_{2}(\F_{q}))$
does not split over $K$. In this regard, the following result of
Gasch\"utz \cite{Gaschutz} will be useful:
\begin{lem}
\label{lem:Gashutz}Let $G$ be a finite group and $A$ an abelian
normal $p$-subgroup of $G$. Then $G$ splits over $A$ if and only
if a Sylow $p$-subgroup $P$ of $G$ splits over $A$.
\end{lem}

Let $P$ be the pre-image under $\rho$ of the upper uni-triangular
subgroup $U_{1}$ in $\SL_{n}(\F_{q})$; then $P$ is a Sylow $p$-subgroup
of $\SL_{n}(W_{2}(\F_{q}))$.

For any $1\leq i,j\leq n$, $i\neq j$, let $e_{ij}=e_{ij}(1)$ denote
the $(i,j)$-elementary unipotent matrix in $\SL_{n}(\F_{q})$ and
let $E_{ij}=E_{ij}(1)$ denote the $(i,j)$-elementary nilpotent matrix
in the Lie algebra $\sl_{n}(\F_{q})$. Note that $e_{ij}=I+E_{ij}$.
We will also consider $e_{ij}$ and $E_{ij}$ as elements in $\SL_{n}(W_{2}(\F_{q}))$
and $\sl_{n}(W_{2}(\F_{q}))$, respectively.

The point behind the following result is to consider possible lifts
of $e_{12}$ of order $p$. We learnt this idea from Y.~de~Cornulier. 
\begin{thm}
\label{thm:non-splitting}Assume that $p\geq5$. Then the element
$e_{12}\in\SL_{n}(\F_{q})$ has no lift to an element in $\SL_{n}(W_{2}(\F_{q}))$
of order $p$. Thus $\SL_{n}(W_{2}(\F_{q}))$ does not split over
$K$.
\end{thm}

\begin{proof}
For simplicity, write $A=e_{12}\in\SL_{n}(W_{2}(\F_{q}))$. Any lift
in $\SL_{n}(W_{2}(\F_{q}))$ of $e_{12}\in\SL_{n}(\F_{q})$ is of
the form
\[
A+pX,
\]
for some $X\in\M_{n}(W_{2}(\F_{q}))$ (note that $p$ is a generator
of the maximal ideal of $W_{2}(\F_{q})$ and that $pX$ only depends
on the image of $X$ in $\M_{n}(\F_{q})$). Since $p^{2}=0$ in $W_{2}(\F_{q})$,
and $[A^{m},X]=m[E_{12},X]=m[A,X]$ for each $m\in\N$, we have

\begin{align*}
(A+pX)^{p} & =A^{p}+\sum_{i=0}^{p-1}A^{i}pXA^{p-1-i}=A^{p}+\sum_{i=0}^{p-1}p(A^{p-1}X-A^{i}[A^{p-1-i},X])\\
 & =A^{p}-p\sum_{i=0}^{p-1}A^{i}[A^{p-1-i},X]=A^{p}-p\sum_{i=0}^{p-1}A^{i}(p-1-i)[A,X]\\
 & =A^{p}+p\Big(\sum_{i=0}^{p-1}A^{i}+iA^{i}\Big)[A,X]\\
 & =A^{p}+p\Big(\sum_{i=0}^{p-1}(1+i)(I+iE_{12})\Big)[A,X]\\
 & =A^{p}+p\Big(\frac{p(p+1)}{2}I+\frac{p(p-1)(p+1)}{3}E_{12}\Big)[A,X].
\end{align*}
Thus, since $p\geq5$, we thus conclude that $(A+pX)^{p}=A^{p}=I+pE_{12}\neq I$,
which proves the first assertion. Now, if $\SL_{n}(W_{2}(\F_{q}))$
splits over $K$, then we would have an injective homomorphism $\SL_{n}(\F_{q})\rightarrow\SL_{n}(W_{2}(\F_{q}))$
such that $e_{12}\in\SL_{n}(\F_{q})$ maps to a lift of $e_{12}$
of order $p$. Since this cannot be the case, $\SL_{n}(W_{2}(\F_{q}))$
does not split over $K$.
\end{proof}

\begin{rem}
\label{rem:Theorem-also-holds}Theorem~(\ref{thm:non-splitting})
also holds for $(n,p)=(2,2)$. Indeed, in this case the proof of the
above theorem shows that $(A+pX)^{2}=A^{2}+2[A,X]=I$ implies that
$X\equiv\begin{pmatrix}x & y\\
0 & x+1
\end{pmatrix}\mod2$, for some $x,y\in\F_{q}$. Then $\det(A+2X)=3\neq1$ in $W_{2}(\F_{q})$,
so there is no lift $A+pX$ in $\SL_{n}(W_{2}(\F_{q}))$ of order
$2$.

On the other hand, for $(n,p)=(3,2)$ there \emph{is} a lift of $e_{12}$
of order $2$, namely
\[
I+E_{12}+2\begin{pmatrix}1 & 0 & 0\\
0 & 0 & 0\\
0 & 0 & -1
\end{pmatrix}.
\]
The square of this element is $I$ mod $4$. Note that this element
has determinant $1$ mod $4$. Similarly, for $n=p=3$ the element
\[
I+E_{12}+E_{23}+3\begin{pmatrix}0 & 0 & 0\\
-1 & -1 & 0\\
0 & 0 & 0
\end{pmatrix}.
\]
has determinant $1$ and cube equal to $I$. It can also be seen that
for $p=3$ this example can be adapted to a lift of $e_{12}$ of order
$3$ for any $n\geq3$. However, in general, for $p=2,3$ a different
approach than the proof of Theorem~(\ref{thm:non-splitting}) is
needed to establish the splitting/non-splitting of $\SL_{n}(W_{2}(\F_{q}))$
(see \cite[Theorem~7]{Sah-I}).
\end{rem}

\bibliographystyle{alex}
\bibliography{alex}

\end{document}